\documentclass[hidelinks,onefignum,onetabnum]{siamart251104}

\usepackage{multirow}
\usepackage{tikz-cd}
\usepackage{subcaption}
\usepackage{amssymb}
\usepackage{algorithm}
\usepackage{algpseudocode}
\usepackage{mydef}
\usepackage{amssymb}
\usepackage{amsmath}
\usepackage{bbm}
\usepackage{ulem}
\usepackage{todonotes}
\usepackage{xspace}

\newsiamremark{remark}{Remark}
\newsiamremark{hypothesis}{Hypothesis}
\crefname{hypothesis}{Hypothesis}{Hypotheses}
\newsiamthm{claim}{Claim}
\newsiamremark{fact}{Fact}
\crefname{fact}{Fact}{Facts}
\algrenewcommand\algorithmicrequire{\textbf{Input:}}
\algrenewcommand\algorithmicensure{\textbf{Output:}}

\usepackage{comment}

\headers{Positivity-Preserving Low-Rank Methods}{K. Kormann, M. Nazarov, J. Wen}

\title{Positivity-preserving dynamical low-rank methods for the Vlasov equation\thanks{Authors are listed alphabetically.
    \funding{This material is based upon work supported by the Swedish Research Council (VR) under grant numbers 2021-04620 and 2021-05095. K.K. is supported by the German Science Foundation DFG under project 530709913.}} }
\author{Katharina Kormann\thanks{Department of Mathematics, Ruhr-Universität Bochum, Universitätsstraße 150, D-44801 Bochum, Germany
    (\email{k.kormann@rub.de})}
  \and Murtazo Nazarov\thanks{Division of Scientific Computing, Department of Information Technology, Uppsala University, Uppsala 751 05 Sweden. 
    (\email{murtazo.nazarov@uu.se}).}
  \and Junjie Wen\thanks{Correspondent author. Division of Scientific Computing, Department of Information Technology, Uppsala University, Uppsala 751 05 Sweden. (\email{junjie.wen@it.uu.se}).}}

\usepackage{amsopn}

\begin{document}

\maketitle

\begin{abstract} 
In this manuscript, we introduce positivity-preserving correction methods for low-rank approximations of the Vlasov equation. The key idea is to formulate structural properties, including positivity-preservation, as constraints and to seek a minimal correction term that is added to the low-rank solution, by solving a quadratic programming problem. As a result, the corrected solution satisfies the constraints and preserve these properties, while remaining close to the original low-rank solution.
Two positivity-preserving schemes are proposed in this work, and one of them also preserves the total mass and momentum of the system. We apply the proposed methods to a Vlasov--Poisson and Vlasov--Poisson-BGK employing a spectral discretization in space and an explicit Runge--Kutta scheme in time. Numerical experiments demonstrate the effectiveness of the proposed methods.
\end{abstract}





\section{Introduction}

The kinetic description of plasma dynamics is crucial for studying plasma physics and understanding various plasma phenomena. The plasma is modeled by a distribution function in phase-space that evolves in electromagnetic fields. This evolution is descriped for collisionless plasmas by the so-called Vlasov equation that is coupled to Maxwell's equation for evolution of the self-consistent fields. For electromagnetic phenomena, the Vlasov--Poisson model is common where the electric field is modelled by a Poisson problem. If collisions play a role, collision operators like the BGK or more sophisticated models are added to the Vlasov equation. Developing accurate and reliable numerical approximations for Vlasov equations has been an important but non-trivial task, and one commonly recognized challenge is due to the high dimensionality of the phase space, which leads to the so-called curse of dimensionality. Low-rank approximations for the Vlasov--Poisson equation have attracted great interest and there are serveral works have been introduced, including the tensor train framework \cite{MR3376137}, the Progressive Generalized Decomposition method \cite{MR3632657}, and the SVD-type truncation \cite{MR4389746}.

In addition to the above mentioned schemes, dynamical low-rank methods using the projector-splitting integrator have been introduced and applied for the Vlasov--Poisson equation; see \cite{MR3177960,MR3863075}. The dynamical low-rank (DLR) method is effective for the Vlasov equations especially in the high-dimensional setttings, as it avoids full phase-space storage and reduces the computational cost. In this work, we employ a rank-adaptive low-rank integrator introduced in \cite{MR4512602}, and we use a standard discretization with spectral methods in space and an explicit Runge--Kutta scheme in time. We use this framework as a testbed to introduce positivity-preserving correction methods, which are applicable in more general settings.

Preserving the non-negativity of the distribution function is a fundamental requirement for numerical approximations of the Vlasov equations. Consequently, the development of positivity-preserving numerical methods has attracted considerable attention. For the full-rank setting, a wide range of positivity-preserving schemes has been proposed based on various spatial discretizations, most notably discontinuous Galerkin (DG) and finite volume (FV) methods; see, \eg \cite{MR2806222, MR2843721, MR4329985, MR3227274, MR4844786, zhang2025positivitypreservinghighordersemilagrangianspectral}. We note that the core limiting algorithm employed in many of these works was originally introduced in \cite{MR2601091}.

In contrast, the literature on positivity-preserving continuous finite element (FE) methods for the Vlasov equations remains relatively limited. We refer the reader to the recent work of \cite{SISC2026}, where an artificial-viscosity-based approach was developed to construct a positivity-preserving continuous finite element approximation of the Vlasov equations. 

However, these techniques do not readily extend to dynamical low-rank approximations. Indeed, the design of high-order positivity-preserving low-rank methods remains an open problem, primarily because suitable low-order positivity-preserving schemes that can serve as the basis for high-order limiting strategies are currently unavailable; see, \eg \cite{MR4925872}.

One notable exception is the approach proposed in \cite{Ye_Loureiro_2024}, where the distribution function $f$ is defined as $f = g^2$ to ensure its positivity and the dynamic low-rank approximation is applied to $g$ instead of $f$. However, this formulation is restricted to a first-order differential operator. To the best of our knowledge, no general positivity-preserving framework has been developed for dynamical low-rank approximations of the Vlasov equations. Addressing this gap is the primary objective and contribution of the present work. One major difficulty of developing positivity-preserving low-rank approximations for the Vlasov equations lies in the fact that the distribution function is not computed explicitly but is decomposed into low-rank tensors, and the connection between the positivity of $f$ and the low-rank factors is not well established yet.  

In addition to positivity, other structural properties of the Vlasov equations, such as the conservation of mass, momentum, and energy, are of fundamental importance. Preserving these invariants is considerably more challenging in dynamical low-rank approximations because the low-rank truncation inherent in the method generally destroys the corresponding conservation properties; see, for example, \cite{MR4728768} introduces a low-rank FE framework that does not conserve the total mass. By contrast, for full-rank approximations, it is relatively straightforward to construct mass conservative FE schemes; see, \eg, \cite{MR4907486,MR4945433}. However, these conservation properties are generally lost in dynamical low-rank methods unless additional corrective procedures are incorporated.

A conservative dynamical low-rank framework was proposed in \cite{MR4280248} and further developed in \cite{MR4569231}, where the authors exploited the underlying continuity equations and introduced suitable correction terms to enforce the desired conservation laws. This approach demonstrates that conservation can be recovered within the low-rank setting, albeit at the expense of additional correction steps.

In this work, we introduce an optimization-based correction technique for dynamical low-rank methods and apply it to the Vlasov equation to enforce positivity preservation and other conservation properties. Let $\tilde{f}$ be the low-rank solution obtained by a standard numerical method, and let $f$ be a corrected low-rank solution satisfying the desired properties, such as posivity-preserving, i.e., $f\geq0$. We seek the minimal correction term $\delta f$: $f=\tilde{f}+\delta f$, by solving the optimization problem: minimize $\Vert\delta f\Vert$, subject to $f\geq 0$. We introduce two correction methods based on this idea and both of them are positivity-preserving; in addition to positivity-preserving, the second one preserves the total mass and momentum while the first introduces a smaller overhead. We will consider the Vlasov--Poisson system in this work but, since the correction is applied in a separate correction step after each time step, it is more general. Similar ideas can be found in \cite{MR5048905,liu2026efficientadmissiblesetprojection}, in which optimization-based limiters are introduced to preserve invariant-domain-preserving for gas dynamics equations and ideal MHD but not applied in the context of dynamical low-rank compression. 

The rest of this manuscript is organized as follows: in Section~\ref{sec:LRVP}, we introduce the Vlasov--Poisson equation and a low-rank approximation for it. In Section~\ref{sec:PPKL}, we introduce a positivity-preserving correction method for the low-rank approximations of the Vlasov equations, and we also introduce the active-set method for solving the optimization problem. In Section~\ref{sec:MMCPPS}, we introduce a mass and momentum conservative as well as positivity-preserving correction method for the low-rank approximations of the Vlasov equations. All the numerical results are presented in Section~\ref{sec:EXP}, together with the details of the spatial and temporal discretizations. The contributions and limitations of this manuscript, and future works are discussed in Section~\ref{sec:CON}.
\section{A low-rank Vlasov--Poisson framework}\label{sec:LRVP}
In this section, we introduce the Vlasov--Poisson equation and a dynamical low-rank framework for it.
\subsection{The Vlasov--Poisson equation} Let $\Omega_{\pmb{x}} \subset \polR^d$ and $\Omega_{\pmb{v}} \subset \polR^d$ be the spatial and velocity domains, respectively, where $d=1,2,3$. In the phase space $\Omega := \Omega_{\pmb{x}} \times \Omega_{\pmb{v}}$, the Vlasov--Poisson equations describing the evolution of the distribution function $f_s(\pmb{x},\pmb{v},t)$ for a species $s$ with charge $q_s$ and mass $m_s$ in a collisionless plasma are given by
\begin{equation}\label{eq:vp}
  \partial_t{f_s}+\pmb{v}\cdot\nabla_{\pmb{x}} f_s+\frac{q_s}{m_s}\pmb{E}\cdot\nabla_{\pmb{v}} f_s=0,
\end{equation}
where $\pmb{E}(\pmb{x},t)= -\nabla_{\pmb{x}}\Phi(\pmb{x},t)$ is the self-consistent electric field, which can be obtained by solving Poisson's equation
\begin{equation}\label{eq:poisson}
  -\nabla^2_{\pmb{x}}\Phi = \rho,
\end{equation}and the charge density $\rho(\pmb{x},t)$ is given by
\begin{equation}\label{eq:charge}
  \rho = \sum_s\int_{\Omega_{\pmb{v}}} q_sf_s\ {\rm d}\pmb{v}.
\end{equation}

For simplicity, we only consider the one-species Vlasov--Poisson system in a neutralizing background in the rest of the paper, and we drop the subscript $s$ for the distribution function, charge, and mass. 

\subsection{Low-rank projectors}
We employ a low-rank representation of the distribution function, as introduced in \cite{MR3863075}. Let us introduce $X_i(\px,t)$, $V_j(\pv,t)$, $i,j=1,\ldots,r$, and the coefficient matrix $\bS = (S_{ij}(t))_{i,j=1}^r$, where $X_i\in L^2(\Omega_{\px})$, $V_j\in L^2(\Omega_{\pv})$, and $S_{ij}\in\polR$, $i,j=1,\ldots,r$. The distribution function is approximated by $f_h$:
\begin{equation}\label{eq:dlrdecomp}
  f_h(\px,\pv,t)=\sum_{i,j=1}^{r} X_i(\px,t) S_{ij}(t)V_j(\pv,t),\notag
\end{equation}
and computing the time derivative of $f_h$, we obtain:
\begin{equation}\label{eq:dlrDerivative}
  \partial_t f_h = \sum_{i,j=1}^r\left(\dot{X}_iS_{ij}V_j+X_i\dot{S}_{ij}V_j+X_iS_{ij}\dot{V}_j\right).\notag
\end{equation}
For uniqueness of the low-rank decomposition, we impose the orthonormality conditions $(X_i,X_j)_{\px}= \delta_{ij}$ and $(V_i,V_j)_{\pv}=\delta_{ij}$, together with the gauge conditions $(X_i,\dot{X}_j)_{\px}=0$ and $(V_i,\dot{V}_j)_{\pv}=0$, for $i,j=1,\ldots,r$, 
where $(\cdot,\cdot)_{\px}$ and $(\cdot,\cdot)_{\pv}$ are the inner products on $L^2(\Omega_{\px})$ and $L^2(\Omega_{\pv})$, respectively. In addition, let the coefficient matrix $\bS$ remain of full rank $r$.

Hence, our approximations belong for $r\geq 1$ to the manifold of rank-$r$ matrices
\begin{equation}
\begin{aligned}
\mathcal{M}_r=\Bigl\{g\in L^2(\Omega_{\px}\times\Omega_{\pv}):g=\sum_{i,j=1}^r X_i S_{ij} V_j,\ \text{with } \bS \text{ invertible}, \\ 
X_i\in L^2(\Omega_{\px}), V_j\in L^2(\Omega_{\pv}), (X_i,X_j)_{\px}= \delta_{ij}, (V_i,V_j)_{\pv}= \delta_{ij}\Bigr\},
\end{aligned}\notag
\end{equation}
and the corresponding tangent space is given by
\begin{equation}
\begin{aligned}
\mathcal{T}_f\mathcal{M}_r=\Bigl\{\dot{g}\in L^2(\Omega_{\px}\times\Omega_{\pv}):\dot{g}=\sum_{i,j=1}^r
\bigl(\dot{X}_i S_{ij} V_j + X_i \dot{S}_{ij} V_j + X_i S_{ij} \dot{V}_j\bigr),\ \text{with }\\ \bS \text{ invertible}, \dot{X}_i\in L^2(\Omega_{\px}), \dot{V}_j\in L^2(\Omega_{\pv}), (X_i,X_j)_{\px}= \delta_{ij}, (V_i,V_j)_{\pv}= \delta_{ij}
\Bigr\}.
\end{aligned}\notag
\end{equation}
We define the associated subspaces
\begin{equation}
  \calX={\rm span}\{X_i\}_{i=1}^r,\qquad\calV={\rm span}\{V_j\}_{j=1}^r,\notag.
\end{equation}
Dynamical low-rank approximations seek solutions propagated over time on the low-rank manifold \cite{koch2007dynamical}. A first robust version was the projector splitting integrator \cite{MR3177960} which evolved the building blocks of the decomposition \eqref{eq:dlrdecomp} separately by decomposing the projector onto the tangent space $\calT_f\mathcal{M}_r$ as follows
\begin{equation}
  P_{\calT_f\calM} = P_{\calV}-P_{\calX\calV}+P_{\calX}. \notag
\end{equation}
In this work, for handling the split projectors, we employ the basis-update-and-Galerkin (BUG) integrator, originally introduced in \cite{MR4375023}, together with its rank-adaptive version proposed in \cite{MR4512602}. We will briefly summarize the algorithm in the following subsection.


\subsection{The rank-adaptive basis-update-and-Galerkin integrator}
Define $K_j = \sum_{i=1}^r X_i S_{ij}$ and $L_i = \sum_{j=1}^r V_j S_{ij}$, $i,j=1,\ldots,r$, and it then follows naturally that $f_h = \sum_{j=1}^r K_j V_j= \sum_{i=1}^r X_i L_i$. 
Let $\bK^m$, $\bL^m$, $\bS^m$ given approximations at time $t_m$.  The BUG integrator splits the integration into three steps: An evolution of the $\bK$ matrix through freezing $\bV$ at $\bV^m$ and multiplying \eqref{eq:dlrDerivative} from the left by $(\bV^m)^\top$ yielding the evolution equation
\begin{equation}\label{eq:Kstep}
  \partial_t K_j= -\sum_{l=1}^r\left(\pmb{c}_{jl}^{1,m}\cdot\nabla_{\pmb{x}} K_l + (\pmb{c}_{jl}^{2,m}\cdot\pmb{E})K_l\right),
\end{equation}
where
\begin{align}
  \pmb{c}_{jl}^{1,m} = (V_j^m,\pmb{v} V_l^m)_{\pv},\qquad\pmb{c}_{jl}^{2,m} = (V_j^m,\nabla_{\pmb{v}} V_l^m)_{\pv}.
\end{align}
Analogously, we obtain the evolution equation for $L$ by freezing $\bU^m$
\begin{equation}\label{eq:Lstep}
    \partial_t L_i= -\sum_{k=1}^r\left(\pmb{d}_{ik}^{1,m}\cdot\nabla_{\pmb{v}} L_k + (\pmb{d}_{ik}^{2,m}\cdot\pmb{v})L_k\right),
\end{equation}
where
\begin{align}
  \pmb{d}_{ik}^{1,m} = (X_i^m,\pmb{E} X_k^m)_{\px},\qquad\pmb{d}_{ik}^2 = (X_i^m,\nabla_{\pmb{x}} X_k^m)_{\px}.\notag
\end{align}
Finally, fixing $\bX$ and $\bV$ at the new time $t_{m+1}$ yields the evolution equation of $\bS$
\begin{equation}\label{eq:Sstep}
    \partial_t S_{ij}=-\sum_{k,l=1}^r(\pmb{c}_{jl}^1\cdot\pmb{d}_{ik}^{2,m+1}+\pmb{c}_{jl}^{2,m+1}\cdot\pmb{d}_{ik}^{1,m+1})S_{kl}.
\end{equation}

We discretize $\Omega_{\px}$ and $\Omega_{\pv}$ using $N_x$ and $N_v$ quadrature nodes, respectively. Let $\{X_{j,i}\}_{j=1}^{N_x}$ and $\{V_{j,i}\}_{j=1}^{N_v}$ denote the degrees of freedom associated with $X_i$ and $V_i$, respectively. For ease of subsequent discussion, we denote the low-rank solutions in the matrix form as follows: let
\begin{equation}
\bX = \big[\pX_1,\ldots,\pX_r\big],\qquad 
\bV = \big[\pV_1,\ldots,\pV_r\big],\notag
\end{equation}
where
\begin{equation}
\pX_i = [X_{1,i},\ldots,X_{N_x,i}]^\mathsf{T}, \qquad 
\pV_i = [V_{1,i},\ldots,V_{N_v,i}]^\mathsf{T}, 
\qquad i=1,\ldots,r.\notag
\end{equation}
Similarly, define
\begin{equation}
\bK = \big[\pK_1,\ldots,\pK_r\big], \qquad 
\bL = \big[\pL_1,\ldots,\pL_r\big], \notag
\end{equation}
where
\begin{equation}
\pK_i = [K_{1,i},\ldots,K_{N_x,i}]^\mathsf{T}, \qquad 
\pL_i = [L_{1,i},\ldots,L_{N_v,i}]^\mathsf{T}, 
\qquad i=1,\ldots,r,\notag
\end{equation}
and $\{K_{j,i}\}_{j=1}^{N_x}$ and $\{L_{j,i}\}_{j=1}^{N_v}$ denote the corresponding degrees of freedom of $K_i$ and $L_i$.
According to the definition of the low-rank operators, the following relations hold:
\begin{equation}
\bK = \bX\bS, \qquad\bL = \bV(\bS)^\mathsf{T}, \qquad\bF = \bX\bS(\bV)^\mathsf{T},\notag
\end{equation}
where $\bF$ contains all degrees of freedom of $f_h$. The QR decomposition is widely used on $\bK$ and $\bV$ to obtain $\bX$ and $\bV$, respectively, as the resulting matrices satisfy
$\bX^\mathsf{T}\bX=\bI_r$ and $\bV^\mathsf{T}\bV=\bI_r$, thereby automatically fulfilling the orthogonality conditions.

 Let $\varepsilon$ be a given tolerance, and for each time step, the integrator algorithm is described in Algorithm~\ref{alg:RAMI}. The spatial and temporal discretizations of the split equations, \eqref{eq:Kstep}, \eqref{eq:Sstep}, and \eqref{eq:Lstep} can be performed by any proper numerical methods, and we utilize the spectral method and explicit Runge-Kutta method as a toy example in our numerical experiments, which will be introduced afterwards.
\begin{algorithm}[htbp]
  \caption{A rank-adaptive matrix integrator}\label{alg:RAMI}
  \begin{algorithmic}[1]
    \Require $\bX^m$, $\bS^m$, $\bV^m$, $\varepsilon$, and $r_{\rm max}$.
    \Ensure $\bX^{m+1}$, $\bS^{m+1}$, $\bV^{m+1}$.
    \State $\bK$-step: Perform \eqref{eq:Kstep} on $\bK^m=\bX^m\bS^m$ and obtain $\bK^{m+1}$.
    \State Compute an orthonormal basis $\overline{\mathbf{X}}$ via the QR decomposition of $[\mathbf{K}^{m+1}, \mathbf{X}^m]$, and set $\mathbf{M} = \overline{\mathbf{X}}^\mathsf{T} \mathbf{X}^m$.
    \State $\bL$-step: Perform \eqref{eq:Lstep} on $\bL^m=\bV^m(\bS^m)^\mathsf{T}$ and obtain $\bL^{m+1}$.
    \State Compute an orthonormal basis $\overline{\mathbf{V}}$ via the QR decomposition of $[\mathbf{L}^{m+1}, \mathbf{V}^m]$, and set $\mathbf{N} = \overline{\mathbf{V}}^\mathsf{T} \mathbf{V}^m$.
    \Comment{Steps 4-5 are performed in parallel with steps 2-3.}
    \State $\bS$-step: Perform \eqref{eq:Sstep} on $\overline{\bS}^m=\bM\bS^m\bN^\mathsf{T}$ and obtain $\overline{\bS}^{m+1}$.
    \State Compute the SVD $\overline{\bS}^{m+1}=\bP\mathbf{\Sigma}\bQ^\mathsf{T}$, where $\mathbf{\Sigma}={\rm diag}(\sigma_1,\ldots,\sigma_{2r})$.
    \State Compute the minimal \( r_1 \leq \min\{2r, r_{\max}\} \) such that $\sqrt{\sum_{j=r_1+1}^{2r}\sigma_j^2}\leq\varepsilon\Vert\Sigma\Vert_F$.
    \State Let $\bS^{m+1}={\rm diag}(\sigma_1,\ldots,\sigma_{r_1})$, $\bX^{m+1}$ and $\bV^{m+1}$ be the first $r_1$ columns of $\bP$ and $\bQ$, respectively.
  \end{algorithmic}
\end{algorithm}

\section{A positivity-preserving low-rank method}\label{sec:PPKL}
In this section, we introduce a positivity-preserving correction method for the low-rank approximations of the Vlasov equations. The main idea is to solve a constrained optimization problem that enforces the positivity as a targeted property. This yields a corrected solution that is positivity-preserving while remaining close to the uncorrected solution.
\subsection{A positivity-preserving correction scheme}
Our first method applies a correction to either $\bK$ or $\bL$ after performing a step with Algorithm \ref{alg:RAMI}. Since the method is symmetic in $\bK$ and $\bL$, we restrict our discussion to $\bK$.

At the discrete level, the positivity-preserving property ensures that the distribution function $f_h$ remains non-negative at the required set of points (e.g., quadrature points or cell averages). In our framework, this is equivalent to
\begin{equation}
   \bF\succeq \pmb{0}_{N_x\times N_v}.\notag
\end{equation}
However, conventional schemes such as positivity-preserving limiters are usually not feasible for the low-rank approximations, as $\bF$ is not solved explicitly but rather decomposed as $\bF = \bK(\bV)^\mathsf{T}$.
We propose the following correction scheme for $\bK$ to achieve preservation of positivity.

Let $\widetilde{\bK}$ and $\bV$ be the solutions obtained by the low-rank schemes after performing  Algorithm \ref{alg:RAMI}, we do the following correction and obtain $\bK$:
\begin{equation}
 \bK=\widetilde{\bK}+\delta\bK,\notag
\end{equation}
such that
\begin{equation}\label{eq:K_pp}
   \bF=\bK(\bV)^\mathsf{T}=(\widetilde{\bK}+\delta\bK)(\bV)^\mathsf{T}\succeq \pmb{0}_{N_x\times N_v},
\end{equation}
where $\delta\bK\in\polR^{N_x\times r}$. The inequality \eqref{eq:K_pp} is equivalent to
\begin{equation}
   \delta\bK(\bV)^\mathsf{T}\succeq-\widetilde{\bK}(\bV)^\mathsf{T}.\notag
\end{equation}
On the other hand, the added $\delta \bK$ results  in a perturbation $\delta\bF$ in our low-rank solution $\widetilde{\bF}=\widetilde{\bK}(\bV)^\mathsf{T}$. Let $\bF_a\in\polR^{N_x\times N_v}$ be the matrix that contains the degrees of freedom of the analytical solution in the chosen basis, and the error of the low-rank methods (before the correction) can be evaluated approximated by
$\Vert\widetilde{\bF}-\bF_a\Vert_F$, where $\Vert\cdot\Vert_F$ denotes the Frobenius norm. Due to the perturbation $\delta\bF$, the error becomes $\Vert\bF-\bF_a\Vert_F$, and
\begin{equation}\label{eq:err_control}
  \Vert\bF-\bF_a\Vert_F=\Vert\delta\bF+\widetilde{\bF}-\bF_a\Vert_F\leq\Vert\widetilde{\bF}-\bF_a\Vert_F+\Vert\delta\bF\Vert_F,
\end{equation}
which implies that the error of the corrected solution is bounded by the error before to the correction and the magnitude of $\|\delta \bF\|_F$. Given that $\bV$ is fixed during this correction step, and 
\begin{equation}
  \Vert\delta\bF\Vert_F=\Vert\delta\bK(\bV)^\mathsf{T}\Vert_F\leq\Vert\delta\bK\Vert_F\Vert\bV\Vert_F,\notag
\end{equation}
$\delta\bK$ is supposed to be as small as possible to minimize the error of the corrected solution. Therefore, the ideal $\delta\bK$ is supposed to satisify \eqref{eq:K_pp} and minimize $\Vert\delta\bK\Vert_F$. Hence, we formulate the above problem into optimization-style as follows:
\begin{equation}\label{eq:OP}
  \begin{aligned}
    &{\rm minimize}&\qquad&\Vert\delta\bK\Vert_F\\
    &{\rm subject\ to}&\qquad& -\delta\bK(\bV)^\mathsf{T}\preceq\widetilde{\bK}(\bV)^\mathsf{T}.
  \end{aligned}
\end{equation}

The optimization problem \eqref{eq:OP} can be decomposed into several subproblems to enhance computational efficiency while preserving the robustness of the solutions.
\begin{proposition}[Row Separability] 
Let $\delta\bK_i\in\polR^{1\times r}$ and $\bb_i\in\polR^{1\times N_v}$ denote the $i$-th rows of $\delta\bK$ and $\widetilde{\bK}(\bV)^\mathsf{T}$, respectively. Solving \eqref{eq:OP} is equivalent to solving the $N_x$ independent problems:
\begin{equation}\label{eq:SOPi}
  \begin{aligned}
    &{\rm minimize}&\qquad&\Vert\delta\bK_i\Vert_2\\
    &{\rm subject\ to}&\qquad& -\delta\bK_i(\bV)^\mathsf{T}\leq\bb_i\Leftrightarrow -\bV(\delta\bK_i)^\mathsf{T}\leq\bb_i^\mathsf{T}.
  \end{aligned}
\end{equation}
\end{proposition}
\begin{proof}
Since the Frobenius norm is separable row-wise and the constraints do not couple the spatial degrees of freedom, the optimization problem \eqref{eq:OP} decomposes into \(N_x\) completely independent optimization subproblems.
\end{proof}

As mentioned earlier, $\mathbf{L}$ has a symmetric structure to $\mathbf{K}$. Hence, the same positivity-preserving correction can be applied analogously to $\mathbf{L}$:
\begin{equation}\label{eq:OP_L}
  \begin{aligned}
    &{\rm minimize}&\qquad&\Vert\delta\bL\Vert_F\\
    &{\rm subject\ to}&\qquad& -\delta\bL(\bX)^\mathsf{T}\preceq\widetilde{\bL}(\bX)^\mathsf{T},
  \end{aligned}
\end{equation}
and we omit the details here.

\subsection{Quadratic programming and active-set method}
The subproblems for \eqref{eq:OP} and \eqref{eq:OP_L} are all instances of quadratic programming (QP) problems. The standard QP problem takes the form
\begin{equation}
  \begin{aligned}\label{eq:QP}
    &{\rm minimize}&\qquad& \frac{1}{2}\bx^\mathsf{T}\bQ\bx + \bc^\mathsf{T}\bx\\
    &{\rm subject\ to}&\qquad& \bA\bx \leq \by, \qquad \bA_{eq}\bx = \by_{eq},
  \end{aligned}
\end{equation}
We take the subproblems \eqref{eq:SOPi} as an example to illustrate the structure of the QP. Let $\bc = \mathbf{0}$, $\bQ = 2\bI_r$, $\bA=-\bV$, $\bA_{eq}=\pmb{0}$, and $\by_{eq}=\pmb{0}$, and for the $i$-th subproblem \eqref{eq:SOPi}, let $\by=\bb_i^\mathsf{T}$, then \eqref{eq:QP} seeks the $\delta\bK_i$ that satisifies $-\bV(\delta\bK_i)^\mathsf{T}\leq\bb^\mathsf{T}_i$ and minimizes $\Vert\delta\bK_i\Vert_2^2$ (so $\Vert\delta\bK_i\Vert_2$).
In addition, the QP problem \eqref{eq:QP} is convex, since $\bQ$ is positive semidefinite, which implies that any local minimum is also a global minimum. We define the Lagrangian function:
\begin{equation}
  \mathcal{L}(\bx,\pmb{\lambda})
  =
  \frac{1}{2}\bx^\mathsf{T}\bQ\bx + \pmb{\lambda}^\mathsf{T}(\bA\bx-\by),\notag
\end{equation}
where $\pmb{\lambda}\in\polR^{N_v\times1}$ are the multipliers associated with the inequality constraints. The Karush--Kuhn--Tucker (KKT) conditions are given by:
\begin{equation}
\begin{aligned}
  \text{Stationarity:} \qquad & \bQ\bx + \bA^\mathsf{T}\pmb{\lambda}= \pmb{0}, \\
  \text{Primal feasibility:} \qquad & \bA\bx \le \by, \\
  \text{Dual feasibility:} \qquad & \pmb{\lambda}\ge\pmb{0},\quad \\
  \text{Complementarity:} \qquad & \pmb{\lambda}^\mathsf{T}(\bA\bx - \by)=0.
\end{aligned}\notag
\end{equation}

\begin{proposition}[Active-Set Sparsity] 
Under non-degeneracy, at most $r$ constraints are strictly active at the optimum for \eqref{eq:QP}.
\end{proposition}
\begin{proof}
The KKT stationarity condition yields $\mathbf{x}^\star = -\frac{1}{2}\mathbf{A}^\mathsf{T}\boldsymbol{\lambda}$.
Thus, \(\mathbf{x}^\star\) lies in the span of the rows of \(\mathbf{A}\) corresponding to indices \(j\) such that \(\lambda_j > 0\), i.e., the active constraints. Since \(\mathbf{x}^\star \in \mathbb{R}^r\), the active set can contain at most \(r\) linearly independent rows of \(\mathbf{A}\).
\end{proof}

There are multiple methods for solving quadratic programming problems; in our case, the active-set method is a promising choice due to the active-set sparsity. Let $\cal$ be the true active set for \eqref{eq:QP}, then $\calS\subset\{1,\ldots,N_v\}$ and $\vert\calS\vert\leq r$. While performing the active-set iteration, the equality-constrained problem yields the
closed-form primal-dual solution:
\begin{equation}\label{eq:CFPD}
  \begin{aligned}
    \pmb{\mu}^\star&=-2(\bA_{\calS}\bA_{\calS}^\mathsf{T})^{-1}\by_{\calS},\\
    \delta\bK^\star&=\frac{1}{2}\bA_{\calS}^\mathsf{T}\pmb{\mu}^\star=-\bA_{\calS}^\mathsf{T}(\bA_{\calS}\bA_{\calS}^\mathsf{T})^{-1}\by_{\calS},
  \end{aligned}
\end{equation}
where $\bA_{\calS}$ and $\by_{\calS}$ are the sub-matrix and sub-vector corresponding to the active indices. The complexity of solving the system \eqref{eq:CFPD} is $O(r^3)$. 

We apply the following workflow for solving \eqref{eq:OP}: for $i=1,\ldots,N_v$, check the condition $\bb_i\geq\pmb{0}$; if true, set $\delta\bK_i=\pmb{0}$. Otherwise, we employ a standard active-set iteration using the closed-form solution \eqref{eq:CFPD} to solve $\eqref{eq:SOPi}$ and obtain $\delta\bK_i$.
The row separability is crucial here, as it reduces the complexity of the optimization problem \eqref{eq:OP} significantly. Without the row separability, \eqref{eq:OP} can still be formulated as a single QP problem in a form similar to \eqref{eq:QP}, but the number of variables increases to $N_xr$, leading to a computational complexity of $O(N_x^3r^3)$ for \eqref{eq:CFPD}, and this is extremely expensive when $r\ll N_x$. In addition, the row separability allows us to solve the independent subproblems in parallel, which can further improve the efficiency of the correction method. However, for this method, imposing global constraints such as conservation of total mass is not trivial, as it breaks the row-separability structure of the scheme.

\subsection{Summary of the algorithm}
In the low-rank approximation of the Vlasov equation, the distribution function is decomposed as $\bF = \bK \bV^{\mathsf{T}} = \bX \bL^{\mathsf{T}}$. We apply the positivity-preserving correction method introduced above to either $\bK$ or $\bL$, thereby rendering the low-rank Vlasov--Poisson framework po\-sit\-ivity-pr\-ese\-rving, and we denote this method as $\mathbf{PP}$-$\mathbf{KL}$. The correction can be applied to either $\bK$ or $\bL$, and we switch the correction between $\bK$ and $\bL$ whenever infeasibility occurs in the optimization problems \eqref{eq:OP} and \eqref{eq:OP_L}. The procedure of the $\mathbf{PP}$-$\mathbf{KL}$ method is described in Algorithm~\ref{alg:PPKL_correction}.

\begin{remark}
For the practical implementation, we employ an active-set method, which allows the inequality constraints to be violated within a prescribed numerical tolerance. To ensure positivity of the numerical solution, we impose the constraint
\begin{equation}
\bF \succeq \epsilon\mathbf{1}_{N_x\times N_v},\notag
\end{equation}
instead, where $\epsilon = 10^{-9}$ is a small positive constant. In addition, the constraint tolerance is chosen to be $0.1\epsilon$, i.e., $10^{-10}$.
\end{remark}
\begin{algorithm}[htbp]
  \caption{\bP\bP-\bK\bL\ method}\label{alg:PPKL_correction}
  \begin{algorithmic}[1]
    \State \textbf{Input:} $\bX^m$, $\bS^m$, $\bV^m$, $\varepsilon$, $r_{\rm max}$, and flag.
    \State \textbf{Output:} $\bX^{m+1}$, $\bS^{m+1}$, $\bV^{m+1}$, and flag.
    \State Perform steps 1-8 of Alg~\ref{alg:RAMI}, obtain $\widetilde{\bX}^{m+1}$, $\widetilde{\bS}^{m+1}$, and $\widetilde{\bV}^{m+1}$.
    \If{flag is 'K'}
    \State Let $\widetilde{\bK}^{m+1}=\widetilde{\bX}^{m+1}\widetilde{\bS}^{m+1}$, solve \eqref{eq:OP}.
    \If{There exists $\delta\bK^{m+1}$}
    \State Let $\bK^{m+1} = \widetilde{\bK}^{m+1}+\delta\bK^{m+1}$, $[\bX^{m+1},\bS^{m+1}]={\rm QR}(\bK^{m+1})$, and $\bV^{m+1}=\widetilde{\bV}^{m+1}$.
    \Else
    \State Change flag to 'L';
    \State Let $\widetilde{\bL}^{m+1}=\widetilde{\bV}^{m+1}(\widetilde{\bS}^{m+1})^\mathsf{T}$, and obtain $\delta\bL^{m+1}$ by solving \eqref{eq:OP_L};
    \State Let $\bL^{m+1} = \widetilde{\bL}^{m+1}+\delta\bL^{m+1}$, $[\bV^{m+1},(\bS^{m+1})^{\mathsf{T}}]={\rm QR}(\bL^{m+1})$, and $\bX^{m+1}=\widetilde{\bX}^{m+1}$.
    \EndIf  
    \Else
    \State Let $\widetilde{\bL}^{m+1}=\widetilde{\bV}^{m+1}(\widetilde{\bS}^{m+1})^\mathsf{T}$, solve \eqref{eq:OP_L}.
    \If{There exists $\delta\bL^{m+1}$}
    \State Let $\bL^{m+1} = \widetilde{\bL}^{m+1}+\delta\bL^{m+1}$, and $[\bV^{m+1},(\bS^{m+1})^\mathsf{T}]={\rm QR}(\bL^{m+1})$, and $\bX^{m+1}=\widetilde{\bX}^{m+1}$.
    \Else
    \State Change flag to 'K';
    \State Let $\widetilde{\bK}^{m+1}=\widetilde{\bX}^{m+1}\widetilde{\bS}^{m+1}$, and obtain $\delta\bK^{m+1}$ by solving \eqref{eq:OP};
    \State Let $\bK^{m+1} = \widetilde{\bK}^{m+1}+\delta\bK^{m+1}$, $[\bX^{m+1},\bS^{m+1}]={\rm QR}(\bK^{m+1})$, and $\bV^{m+1}=\widetilde{\bV}^{m+1}$.
    \EndIf  
    \EndIf
  \end{algorithmic}
\end{algorithm}
\begin{remark}[Feasibility]
We would like to emphasize that in theorem, it is not guaranteed that the optimization problems \eqref{eq:OP} and \eqref{eq:OP_L} have solutions in the low-rank setting, especially when $r\ll N_x, N_v$. In our numerical experiments, we try to find a solution to one of the optimization problems, e.g. \eqref{eq:OP}, and switch to the other as soon as the problem is not solvable. We have not encountered any failure cases with the strategy in our experiments, but this topic is worth investigating in future works. In particular, increasing the rank of the solution might be necessary to ensure feasibility.
\end{remark}

\section{A conservative and positivity-preserving low-rank method}\label{sec:MMCPPS}
As mentioned in the previous section, incorporating global constraints, such as total mass conservation, into the correction method for $\bK$ and $\bL$ is very costly, as these constraints break the row separability, which is crucial for reducing the computational complexity of the optimization problems. However, if we apply the correction to $\bS$, instead of $\bK$ or $\bL$, the global constraints can be imposed together with the positivity-preserving property, as there is no row separability for $\bS$.

In this section, we introduce a mass and momentum conservative as well as positivity-preserving correction technique on \bS\,  which we will refer to as \bM\bM\bC-\bP\bP-\bS. Let $\bX$, $\widetilde{\bS}$, and $\bV$ be the low-rank solutions after one time step with Algorithm \ref{alg:RAMI}, we do the following correction and obtain $\bS$:
\begin{equation}
 \bS=\widetilde{\bS}+\delta\bS,\notag
\end{equation}
where $\delta\bS\in\polR^{r\times r}$ such that the low-rank solution is positivity-preserving:
\begin{equation}
   \bF=\bX\bS(\bV)^\mathsf{T}=\bX(\widetilde{\bS}+\delta\bS)(\bV)^\mathsf{T}\succeq \pmb{0}_{N_x\times N_v},\notag
\end{equation}
and the above inequality is equivalent to
\begin{equation}\label{eq:S_pp}
  \bX\delta\bS(\bV)^\mathsf{T}\succeq-\bX\widetilde{\bS}(\bV)^\mathsf{T}.
\end{equation}

Let $\calM$ and $\calJ$ denote the total mass and momentum of the system, respectively, the following conservation properties hold for the Vlasov--Poisson equations:
\begin{itemize}
  \item \textbf{mass conservation}: $\frac{\rm d}{{\rm d}t}\calM=\frac{\rm d}{{\rm d}t}\left(\int_\Omega f\ {\rm d}\pmb{x}{\rm d}\pmb{v}\right)=0$;
  \item \textbf{momentum conservation}: $\frac{\rm d}{{\rm d}t}\calJ=\frac{\rm d}{{\rm d}t}\left(\int_\Omega\pmb{v}f\ {\rm d}\pmb{x}{\rm d}\pmb{v}\right)=0$.
\end{itemize}
Assuming we use a nodal basis representation on a uniform phase-space discretization with mesh sizes $\Delta x$ and $\Delta v$, the discrete total mass is defined by 
\begin{equation}
  \calM_h=\sum_{i=1}^{N_x}\sum_{j=1}^{N_v}F_{ij}(\Delta x)^d(\Delta v)^d,\notag
\end{equation}
where $F_{ij}$ is the $(i,j)$-th entry of $\bF$. The discrete total momentum is given by 
\begin{equation}
  \calJ_h=\sum_{i=1}^{N_x}\sum_{j=1}^{N_v}\sum_{k=1}^d F_{ij}(v_k)_j(\Delta x)^d(\Delta v)^d.\notag
\end{equation}
Denote by $\mathbf{v} \in \mathbb{R}^{N_v \times 1}$ the vector of discrete velocity points. We omit the term $(\Delta x)^d(\Delta v)^d$ for simplicity, since it is a constant. The discrete total mass and momentum are then given by
\begin{equation}
  \calM_h = \pmb{1}_{1\times N_x}\bX\bS(\bV)^\mathsf{T}\pmb{1}_{N_v\times 1},\qquad
  \calJ_h = \pmb{1}_{1\times N_x}\bX\bS(\bV)^\mathsf{T}\bv.\notag
\end{equation}
Let $\bS$ preserve the total mass and total momentum of the distribution function, i.e.,
\begin{equation}
  \pmb{1}_{1\times N_x}\bX\bS(\bV)^\mathsf{T}\pmb{1}_{N_v\times 1}=\calM_{h,0},\qquad
  \pmb{1}_{1\times N_x}\bX\bS(\bV)^\mathsf{T}\bv=\calJ_{h,0},\notag
\end{equation}
where $\calM_{h,0}$ and $\calJ_{h,0}$ denote the initial discrete total mass and total momentum, respectively. We add the following additional constraint:
\begin{equation}\label{eq:S_mass}
  \pmb{1}_{1\times N_x}\bX\delta\bS(\bV)^\mathsf{T}\pmb{1}_{N_v\times 1}=\calM_{h,0}-\pmb{1}_{1\times N_x}\bX\widetilde{\bS}(\bV)^\mathsf{T}\pmb{1}_{N_v\times 1},
\end{equation}
and 
\begin{equation}\label{eq:S_momentum}
  \pmb{1}_{1\times N_x}\bX\delta\bS(\bV)^\mathsf{T}\bv=\calJ_{h,0}-\pmb{1}_{1\times N_x}\bX\widetilde{\bS}(\bV)^\mathsf{T}\bv.
\end{equation}

Collecting the constraints \eqref{eq:S_pp}, \eqref{eq:S_mass}, and \eqref{eq:S_momentum}, we obtain the following optimization problem
\begin{equation}\label{eq:S_op}
  \begin{aligned}
    &{\rm minimize}&\qquad&\Vert\delta\bS\Vert_F\\
    &{\rm subject\ to}&\qquad& \bX\delta\bS(\bV)^\mathsf{T}\succeq-\bX\widetilde{\bS}(\bV)^\mathsf{T},\\
    &&& \pmb{1}_{1\times N_x}\bX\delta\bS(\bV)^\mathsf{T}\pmb{1}_{N_v\times 1}=\calM_{h,0}-\pmb{1}_{1\times N_x}\bX\widetilde{\bS}(\bV)^\mathsf{T}\pmb{1}_{N_v\times 1},\\
    &&& \pmb{1}_{1\times N_x}\bX\delta\bS(\bV)^\mathsf{T}\bv=\calJ_{h,0}-\pmb{1}_{1\times N_x}\bX\widetilde{\bS}(\bV)^\mathsf{T}\bv.
  \end{aligned}
\end{equation}
\begin{proposition}
  Let $\rm{vec}(\bA)$ denote the vector obtained by stacking the columns of $\bA$ on top of each other, and the following property holds for the $\rm{vec}$ operator:
  ${\rm vec}(\bA\bB\bC)$ = $(\bC^\mathsf{T}\otimes\bA){\rm vec}(\bB)$, where $\otimes$ denotes the Kronecker product. 
\end{proposition}
We can rewrite the optimization problem \eqref{eq:S_op} into the following form:
\begin{equation}\label{eq:QP_S} 
  \begin{aligned}
    &{\rm minimize}&\qquad&\Vert{\rm vec}(\delta\bS)\Vert_2\\
    &{\rm subject\ to}&\qquad& \bM_1{\rm vec}(\delta\bS)\succeq-\bM_1{\rm vec}(\widetilde{\bS}),\\
    &&& \bM_2{\rm vec}(\delta\bS)=\calM_{h,0}-\bM_2{\rm vec}(\widetilde{\bS}),\\
    &&& \bM_3{\rm vec}(\delta\bS)=\calJ_{h,0}-\bM_3{\rm vec}(\widetilde{\bS}).
  \end{aligned}
\end{equation}
where $\bM_1 = (\bV\otimes\bX)$, $\bM_2 = (\pmb{1}_{1\times N_v}\bV)\otimes(\pmb{1}_{1\times N_x}\bX)$, and $\bM_3 = (\bv^\mathsf{T}\bV)\otimes(\pmb{1}_{1\times N_x}\bX)$.

The system \eqref{eq:QP_S} is a QP problem with $r^2$ variables, and its structure is similar to \eqref{eq:SOPi}, except for the inclusion of two additional equality constraints. Since the number of variables is reduced from $N_x r$ (or $N_v r$) to $r^2$, the computational cost remains affordable. Given that we apply the active-set method, the computational complexity for solving \eqref{eq:CFPD} is $O((r^2)^3)=O(r^6)$, which is significantly more efficient than $O(N_x^3r^3)$.

The optimization problem \eqref{eq:QP_S} is generally more likely to be infeasible compared to \eqref{eq:OP} and \eqref{eq:OP_L} because there are additional constraints and the number of variables is reduced. In addition, the switching strategy between $\bK$ and $\bL$ is not applicable here, as the correction is only applied to $\bS$. 
To address this issue, one strategy is to increase the amount of variables whenever we fail to find a solution for \eqref{eq:QP_S}. Nevertheless, we can increase the rank $r$ to enhance the feasibility of \eqref{eq:QP_S}, as the number of variables increases with $r$, and it is trivial within the rank-adaptivity integrator. The workflow of the \bM\bM\bC-\bP\bP-\bS\ method is described in Alg~\ref{alg:MMCPPS_correction}. 

\begin{algorithm}[htbp]
  \caption{\bM\bM\bC-\bP\bP-\bS\ method}\label{alg:MMCPPS_correction}
  \begin{algorithmic}[1]
    \Require $\bX^m$, $\bS^m$, $\bV^m$, $\varepsilon$, and $r_{\rm max}$.
    \Ensure $\bX^{m+1}$, $\bS^{m+1}$, $\bV^{m+1}$.
    \State Perform steps 1-7 of Alg~\ref{alg:RAMI}.
    \While{1}
    \If{$r_1 \leq\min\{r_{\rm max},2r\}$}
    \State Let $\widetilde{\bS}^{m+1}={\rm diag}(\sigma_1,\cdots,\sigma_{r_1})$, $\widetilde{\bX}^{m+1}$ and $\widetilde{\bV}^{m+1}$ be the first $r_1$ columns of $\hat{\bP}$ and $\hat{\bQ}$, respectively.
    \EndIf
    \State Solve \eqref{eq:QP_S} and obtain $\delta\bS^{m+1}$.
    \If{This is no solution for $\delta\bS^{m+1}$}
    \State $r_1 = r_1+1$.
    \Else
    \State Let $\bX^{m+1} = \widetilde{\bX}^{m+1}$, $\bV^{m+1} = \widetilde{\bV}^{m+1}$, $\bS^{m+1} = \widetilde{\bS}^{m+1}+\delta\bS^{m+1}$, and break.
    \EndIf
    \EndWhile
  \end{algorithmic}
\end{algorithm}

\section{Numerical experiments}\label{sec:EXP}
In this section, we present the numerical results of the \bP\bP-\bK\bL\ and \bM\bM\bC-\bP\bP-\bS\ methods to verify the effectiveness of these methods. We employ a toy model based on the 1D1V Vlasov--Poisson equation, using a spectral method for spatial discretization and the classical fourth-order explicit Runge-Kutta method for temporal discretization within the $\bK$, $\bL$ and $\bS$ step, respectively. We verify that both methods can preserve the positivity of the distribution function, and the \bM\bM\bC-\bP\bP-\bS\ method can also preserve the total mass and momentum of the system. In addition, we compare the results of these two methods with those obtained without correction. 

All experiments were performed on a system with an Intel Core i7-13700K processor under Ubuntu. The simulations were implemented using MATLAB R2023b. For all the numerical experiments presented in this section, we use $\varepsilon=1.0^{-6}$ and $r_{\rm max} = 20$.
\subsection{Spatial and temporal discretization} First, we discretize $\Omega_x$ using $N_x$ uniform degrees of freedom and apply a Fourier transform in the $x$-space. We represent $K_j$, $j=1,\ldots,r$, in Fourier space as
\begin{equation}
  K_j(x,t) = \sum_{\alpha=-N_x/2}^{N_x/2} \hat{K}_j(\alpha,t)\exp(\si\alpha x), \notag
\end{equation}
from which it follows that
\begin{equation}
  \partial_x K_j = \sum_{\alpha=-N_x/2}^{N_x/2} (\si\alpha)\hat{K}_j(\alpha)\exp(\si\alpha x). \notag
\end{equation}
The $\mathbf{K}$-step \eqref{eq:Kstep} hence becomes:
\begin{equation}
  \partial_t K_j = -\sum_{l=1}^r\left(c_{jl}^1\sum_{\alpha=-N_x/2}^{N_x/2}(\si\alpha)\hat{K}_l\exp(\si\alpha x)+c_{jl}^2(E_h K_l)\right),\notag
\end{equation}
where $E_h$ is an approximation of the electric field, and the above equation can be replaced by
\begin{equation}
  \partial_t K_j = -\sum_{l=1}^r\left(c_{jl}^1\calF^{-1}_{N_x}(\calF_{N_x}(K_l))+c_{jl}^2E_hK_l\right),\notag
\end{equation}
where $\mathcal{F}_{N_x}$ and $\mathcal{F}_{N_x}^{-1}$ denote the discrete Fourier transform (DFT) over $\Omega_x$ using $N_x$ points and its inverse, respectively. Similarly, we discretize the $\mathbf{L}$-step \eqref{eq:Lstep} as follows:
\begin{equation}
  \partial_t L_i = -\sum_{k=1}^r\left(d_{ik}^1\calF_{N_v}^{-1}(\calF_{N_v}(L_k))+d_{jl}^2vL_k\right).\notag
\end{equation}

For the temporal discretization, we employ the classical fourth-order explicit Runge--Kutta method. For the semi-discrete system $\partial_t \pmb{u} = g(\pmb{u})$, the solution is advanced from $\pmb{u}^m$ to $\pmb{u}^{m+1}$ over a time step $\tau_m$ as follows:
\begin{equation}
    \pmb{k}_1 = g(\pmb{u}^m), \quad
    \pmb{k}_2 = g\!\left(\pmb{u}^m + \tfrac{\tau_m}{2} \pmb{k}_1\right), \quad
    \pmb{k}_3 = g\!\left(\pmb{u}^m + \tfrac{\tau_m}{2} \pmb{k}_2\right), \quad
    \pmb{k}_4 = g\!\left(\pmb{u}^m + \tau_m \pmb{k}_3\right),
\notag
\end{equation}
and the solution is updated by
\begin{equation}
    \pmb{u}^{m+1} = \pmb{u}^m + \frac{\tau_m}{6}\left(\pmb{k}_1 + 2\pmb{k}_2 + 2\pmb{k}_3 + \pmb{k}_4\right).\notag
\end{equation}
The time step is controlled by the following CFL condition
\begin{equation}
  \Delta t = {\rm cfl}\frac{\min(\Delta x, \Delta v)}{\Vert\pmb{\bbetaa}_h\Vert_{\infty},\notag}
\end{equation}
where $\pmb{\bbetaa}_h=(v,E_h)^\mathsf{T}$.  We use ${\rm cfl} = 0.4$ for all the experiments presented in this manuscript.
\subsection{Landau damping}
First, we consider the Landau damping test case, let the phase space be $[0,2\pi/k]\times[-6,6]$, we take the following initial data 
\begin{equation}\label{eq:IN_LD}
 f(x,v,0)=\frac{1}{\sqrt{2\pi}}{\rm exp}\left(-\frac{v^2}{2}\right)(1+\alpha{\rm \cos}(k x)),\notag
\end{equation}
where $\alpha = 0.01$ and $k = 0.5$. When studying the Landau damping using numerical schemes, one common issue is the numerical recurrence, which is a phenomenon where the electric energy recovers to the initial state after a certain time, and the numerical solution becomes unreliable after the recurrence time. This is due to the discretization of the velocity space. 
We compute the electric energy as follows
\begin{equation}
\calE_e = \frac{1}{2}{\int_{\Omega_x}(E_h)^2\ {\rm d}x},\notag
\end{equation} and plot the time evolution of $\calE_e$ in a logarithm scale for the case without correction and for the cases using the \bP\bP-\bK\bL\ and \bM\bM\bC-\bP\bP-\bS\ methods up to the recurrence time, in Fig.~\ref{fig:LD1} (a), together with the expected damping rate. It is observed that the evolution of $\calE_e$ is consistent with the expected damping rate for both cases. Plotting the evolution of the rank $r$ of the solution in Fig.~\ref{fig:LD1} (b), and we observe an initial drop in rank from $r=5$, followed by a increasing to rank $5$, where it subsequently remains constant. In addition, we plot the relative mass error and absolute momentum error for all cases in Fig.~\ref{fig:LD1} (c) and Fig.~\ref{fig:LD1} (d), respectively. In this case, the distribution function remains non-negative throughout the simulation, and the $\bP\bP$-$\bK\bL$ correction adds nothing to the original solution. Consequently, the mass and momentum errors stay the same as those of the uncorrected solution. We see that both the mass and momentum error is significantly reduced and remain negligible after applying the \bM\bM\bC-\bP\bP-\bS\ method.
\begin{figure}[htbp]
  \centering
  \begin{subfigure}[b]{0.45\textwidth}
    \includegraphics[width=\linewidth]{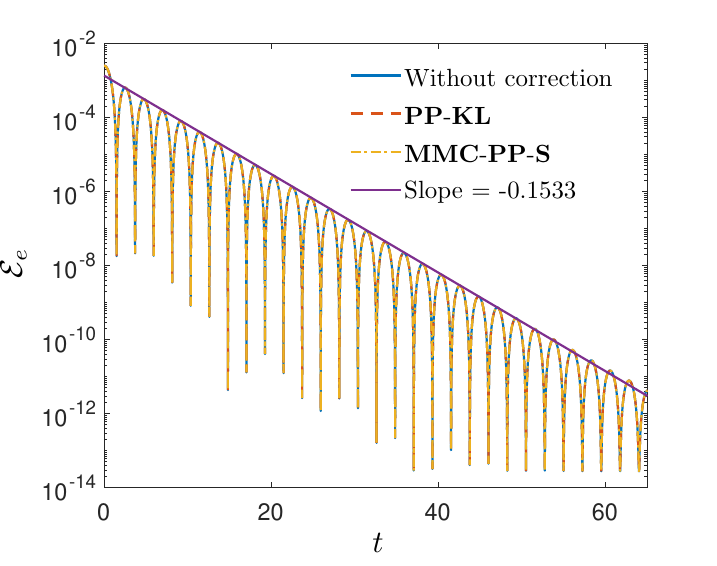}
    \caption{Electric energy}
  \end{subfigure}~
  \begin{subfigure}[b]{0.45\textwidth}
    \includegraphics[width=\linewidth]{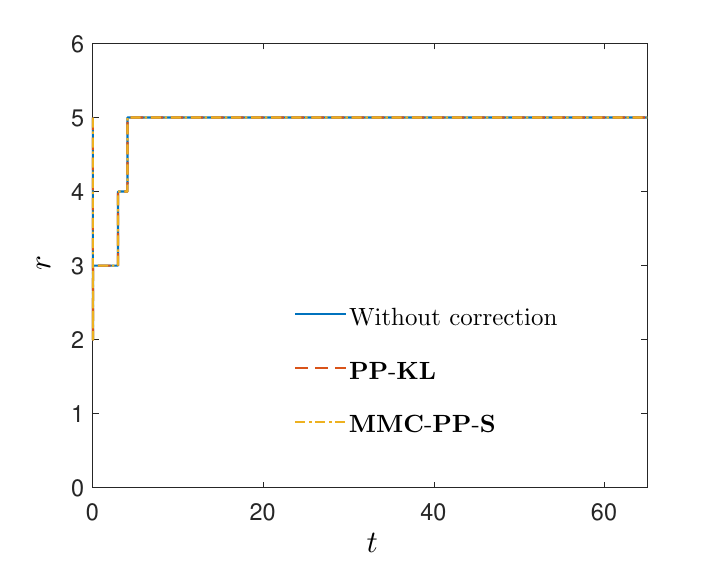}
    \caption{Rank}
  \end{subfigure}
  \begin{subfigure}[b]{0.45\textwidth}
    \includegraphics[width=\linewidth]{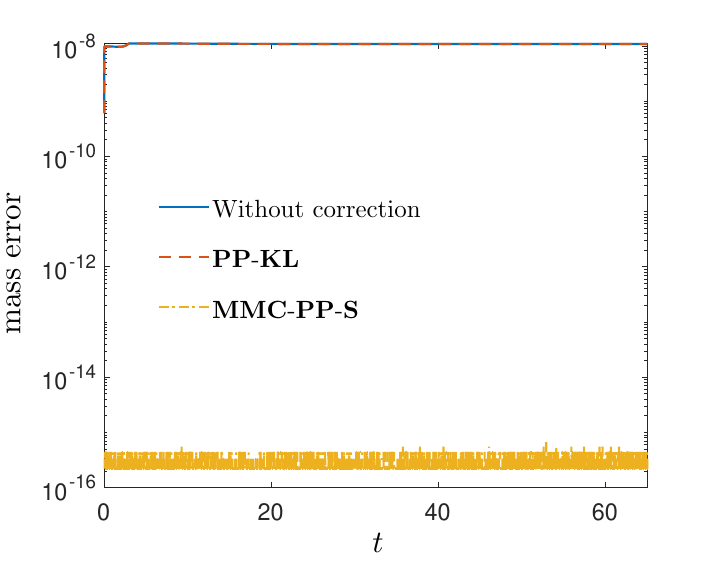}
    \caption{Mass error}
  \end{subfigure}~
  \begin{subfigure}[b]{0.45\textwidth}
    \includegraphics[width=\linewidth]{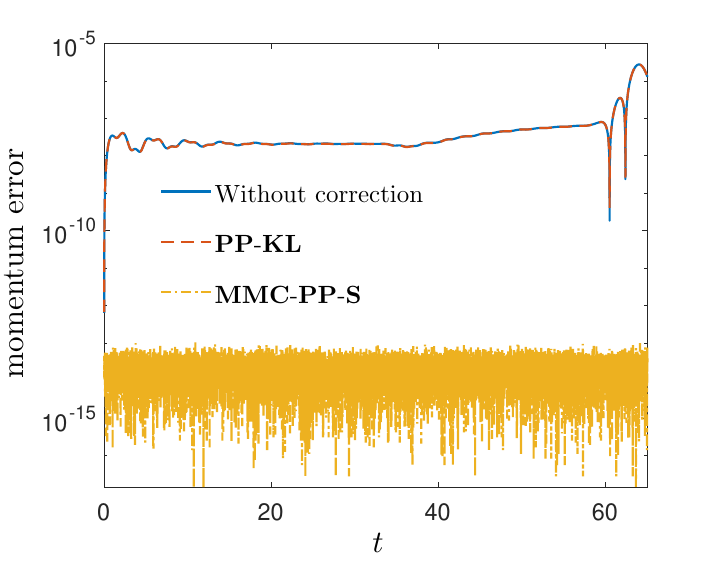}
    \caption{Momentum error}
  \end{subfigure}
  \caption{Landau damping: evolution of the electric energy (a), rank (b), the relative mass error (c), and absolute momentum error (d), respectively. The number of degrees of freedom is $N_x = N_v = 128$.}
  \label{fig:LD1}
\end{figure}

To demonstrate that our scheme is also applicable in the presence of collisions, we consider the BGK model. The Vlasov--Poisson equation with the BGK operator is given by
\begin{equation}
  \partial_t f + v\partial_x f + E\partial_v f = \nu(f_{\rm eq} - f),\notag
\end{equation}
where $\nu$ is the collision frequency, and $f_{\rm eq}$ is the local Maxwellian distribution function defined by
\begin{equation}
  f_{\rm eq}(x,v,t) = \frac{\rho(x,t)}{\sqrt{2\pi T(x,t)}}{\rm exp}\left(-\frac{(v-u(x,t))^2}{2T(x,t)}\right),\notag  
\end{equation}where $\rho(x,t)$, $u(x,t)$, and $T(x,t)$ are the local density, velocity, and temperature, respectively, which are defined by
\begin{equation}
  \rho = \int_\Omega f{\rm d}v,\qquad u = \frac{1}{\rho}\int_\Omega vf{\rm d}v,\qquad T = \frac{1}{\rho}\int_\Omega (v-u)^2f{\rm d}v.\notag
  \end{equation}
The BGK operator is a simplified collision operator that relaxes the distribution function towards a local Maxwellian distribution function, and it is widely used in plasma physics and kinetic theory. We consider the Landau damping with the BGK operator, and set $\nu = 0.01$. We shows plots of the evolution of $\calE_e$, the rank $r$, the relative mass error and absolute momentum error in Fig.~\ref{fig:LD2}. As in the collisionless case, the \bM\bM\bC-\bP\bP-\bS\ method significantly reduces the mass and momentum errors. It is also observed that the presence of the BGK collision operator leads to a slightly larger damping rate than in the collisionless case. This behavior is consistent with the findings in \cite{MR2100514}, where the collisional Fokker--Planck--Landau model is investigated. In particular, the damping rate measured from Fig.~\ref{fig:LD2} (a) is approximately 0.1644. With the same parameters, \cite[Table~2]{MR2100514} presentes a theoretical estimate of 0.154 and a numerical damping rate of 0.168. Our computed value agrees well with the reported numerical result, demonstrating the accuracy of the proposed method for the collisional Vlasov--Poisson--BGK system.
\begin{figure}[htbp]
  \centering
  \begin{subfigure}[b]{0.45\textwidth}
    \includegraphics[width=\linewidth]{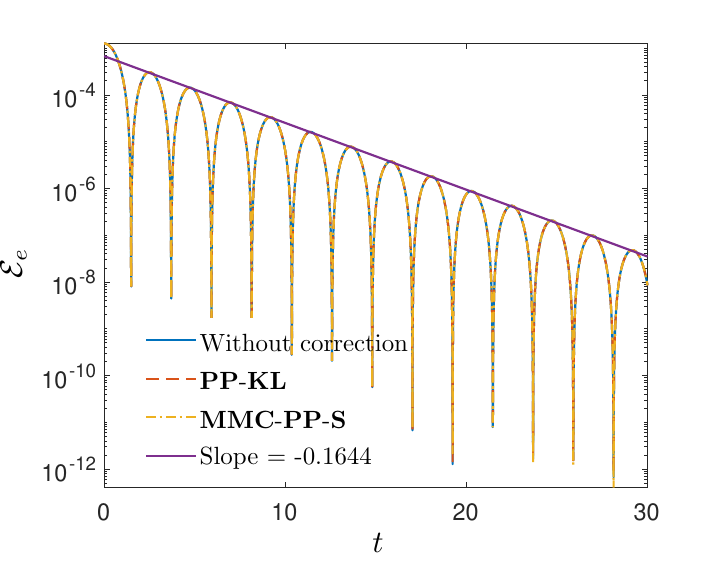}
    \caption{Electric energy}
  \end{subfigure}~
  \begin{subfigure}[b]{0.45\textwidth}
    \includegraphics[width=\linewidth]{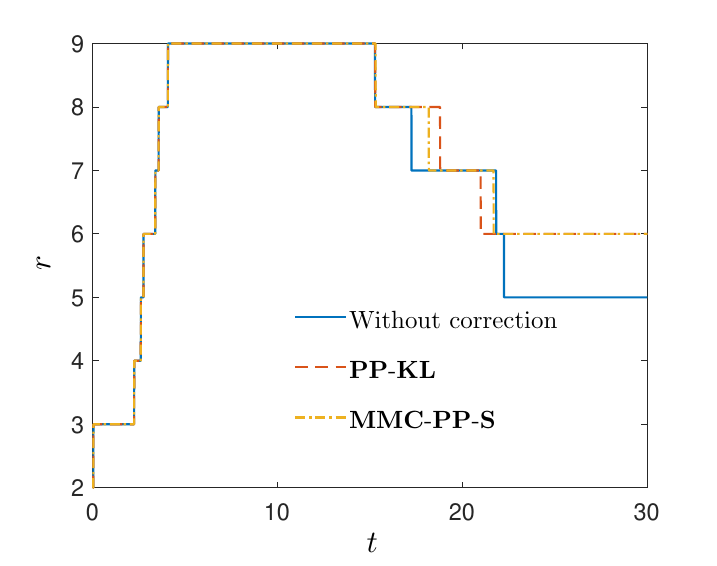}
    \caption{Rank}
  \end{subfigure}
  \begin{subfigure}[b]{0.45\textwidth}
    \includegraphics[width=\linewidth]{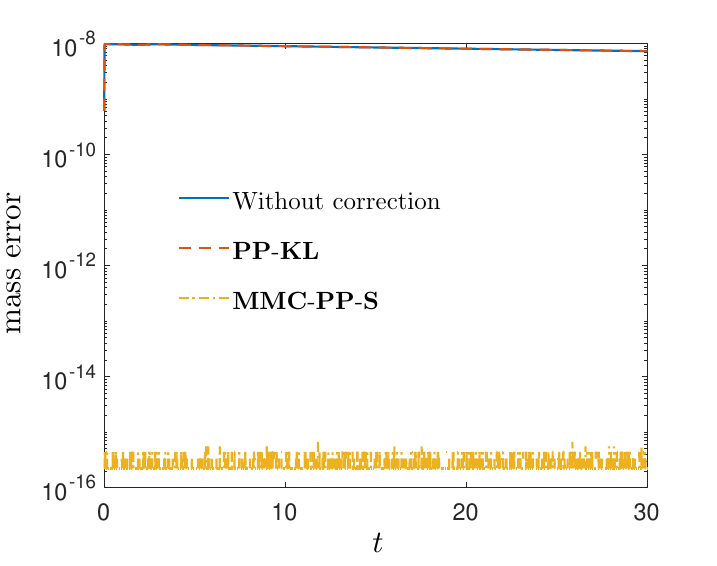}
    \caption{Mass error}
  \end{subfigure}~
  \begin{subfigure}[b]{0.45\textwidth}
    \includegraphics[width=\linewidth]{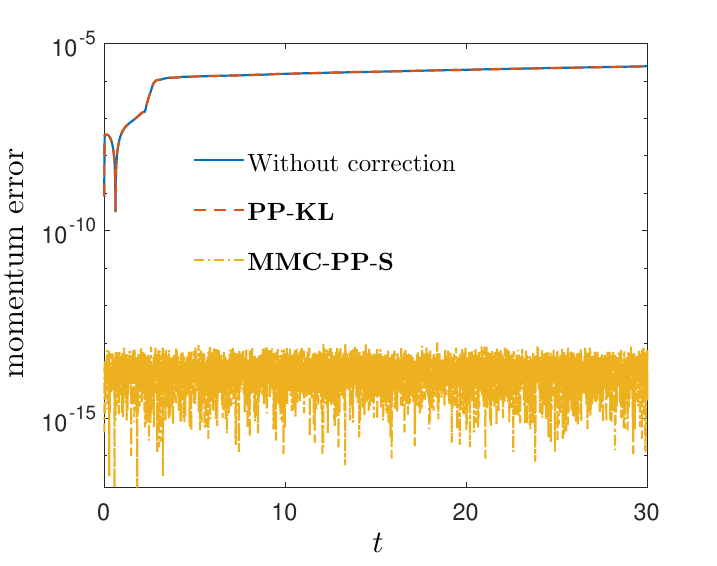}
    \caption{Momentum error}
  \end{subfigure}
  \caption{Landau damping with BGK operator: evolution of the electric energy (a), rank (b), the relative mass error (c), and absolute momentum error (d), respectively. The number of degrees of freedom is $N_x = N_v = 128$.}
  \label{fig:LD2}
\end{figure}
\subsection{Two-stream instability} Next, we consider the two-stream instability, we let the phase space be $[0,2\pi/k]\times[-8,8]$, and use the following initial data
\begin{equation}
 f(x,v,0)=\frac{1}{2\sqrt{2\pi}}{\rm exp}\left(-\frac{(v-2.4)^2}{2}-\frac{(v+2.4)^2}{2}\right)(1+\alpha{\rm \cos}(k x)),\notag
\end{equation}
where $\alpha = 0.001$ and $k = 0.2$. Plot the evolution of $\calE_e$ in Fig.~\ref{fig:TS_fig1} (a), and we see that the evolution of the electric energy is correctly described by the low-rank approximation, in agreement with the results reported in \cite{MR3863075,MR4945433,MR4280248,MR4402737}. Plotting the rank of the distribution function in Fig.~\ref{fig:TS_fig1} (b), we also observe that the rank continues to grow, except for a single drop at the beginning, eventually reaching the prescribed maximum rank and remaining there thereafter.
In addition, we show the relative mass error and absolute momentum error in Fig.~\ref{fig:TS_fig1} (c) and Fig.~\ref{fig:TS_fig1} (d), respectively. It is observed that the \bP\bP-\bK\bL\ method can not preserve the conservation of total mass and momentum, on the contrary the imbalanced is increased compared to the method without correction, while the \bM\bM\bC-\bP\bP-\bS\ method controls both the mass error and momentum error within machine precision.
\begin{figure}[htbp]
  \centering
  \begin{subfigure}[b]{0.45\textwidth}
    \includegraphics[width=\linewidth]{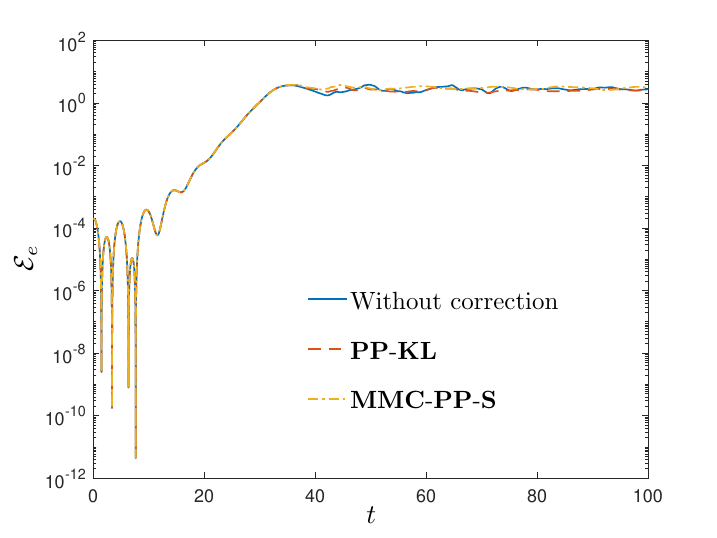}
    \caption{Electric energy}
  \end{subfigure}~
    \begin{subfigure}[b]{0.45\textwidth}
    \includegraphics[width=\linewidth]{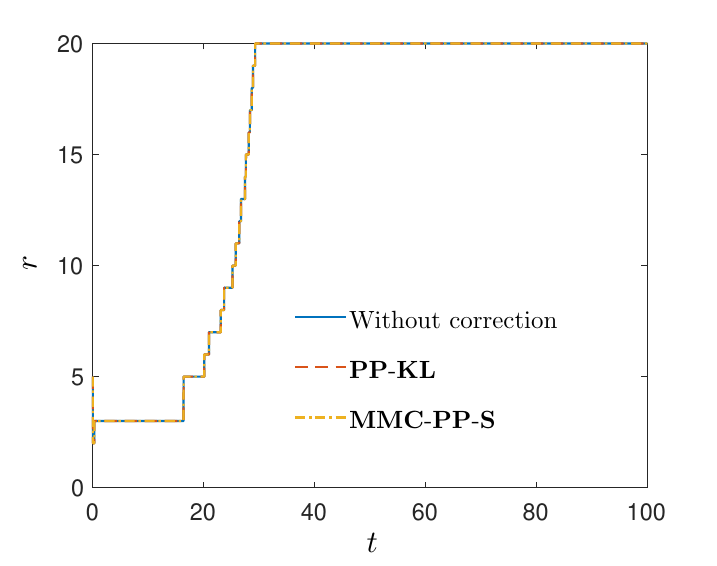}
    \caption{Rank}
  \end{subfigure}
    \begin{subfigure}[b]{0.45\textwidth}
    \includegraphics[width=\linewidth]{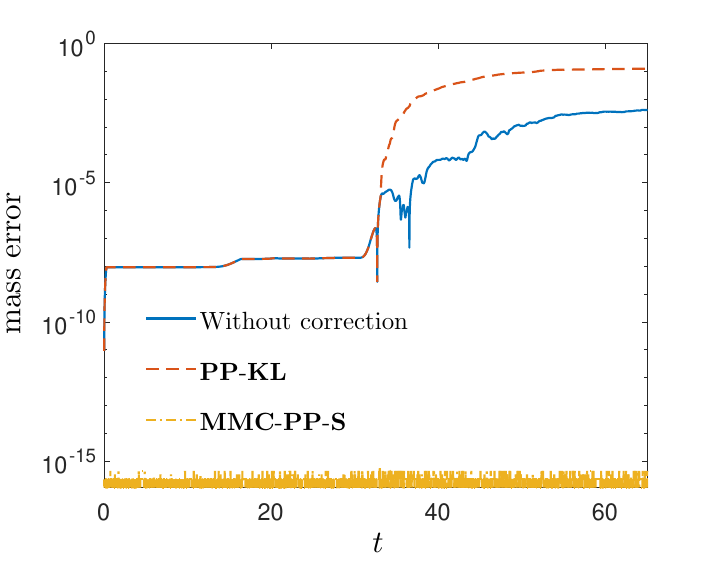}
    \caption{Mass error}
  \end{subfigure}~
    \begin{subfigure}[b]{0.45\textwidth}
    \includegraphics[width=\linewidth]{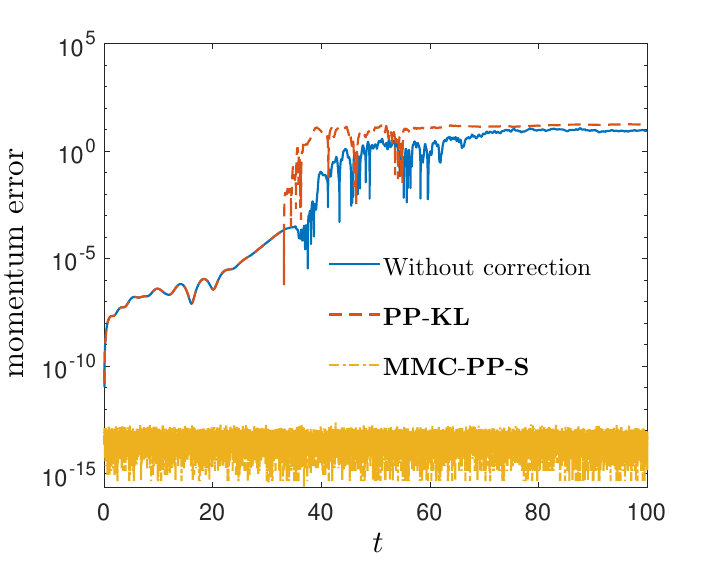}
    \caption{Momentum error}
  \end{subfigure}
   \caption{Two-stream instability: evolution of the electric energy (a), rank (b), the relative mass error (c), and absolute momentum error (d), respectively. The number of degrees of freedom is $N_x = N_v = 128$.}
   \label{fig:TS_fig1}
\end{figure}

For the purpose of verifying the positivity-preserving property, we plot the lower bound of the distribution function in Fig.~\ref{fig:TS_fig2}, and it is observed that the minimal value of the distribution function becomes negative at round $t=30$ without correction, while it remains positive when applying both methods, which implies that the positivity is preserved in both methods. We also show the distribution functions at $t=35$ in Fig~\ref{fig:TS_fig3}, for the case without correction and the cases with the \bP\bP-\bK\bL\ and \bM\bM\bC-\bP\bP-\bS\ methods applied. The results in Fig.~\ref{fig:TS_fig3} are close, despite the fact that the distribution function introduces negative values in Fig.~\ref{fig:TS_fig3} (a) without corrections, while it remains all positive in Fig.~\ref{fig:TS_fig3} (b) and (c) when the \bP\bP-\bK\bL\ or \bM\bM\bC-\bP\bP-\bS\ method applied. 
\begin{figure}[htbp]
  \centering
  \begin{subfigure}[b]{0.45\textwidth}
    \includegraphics[width=\linewidth]{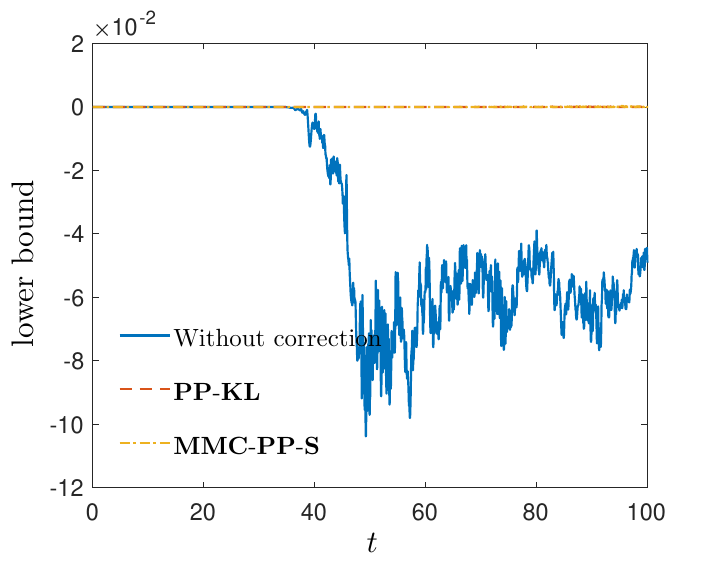}
  \end{subfigure}
   \caption{Two-stream instability: evolution of the lower bound of the distribution function. The number of degrees of freedom is $N_x = N_v = 128$.}
   \label{fig:TS_fig2}
\end{figure}
\begin{figure}[htbp]
  \centering
  \begin{subfigure}[b]{0.45\textwidth}
    \includegraphics[width=\linewidth]{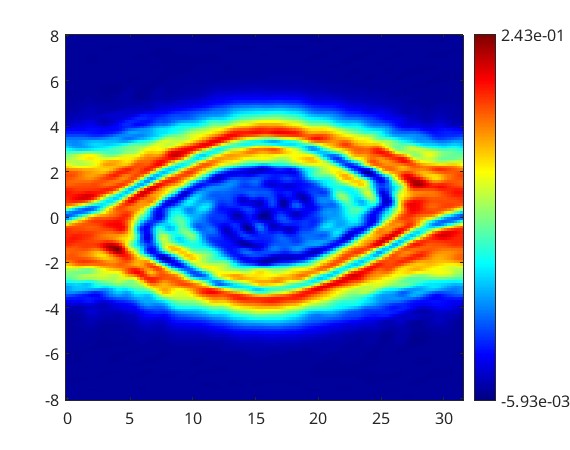}
    \caption{Without correction}
  \end{subfigure}~
  \begin{subfigure}[b]{0.45\textwidth}
    \includegraphics[width=\linewidth]{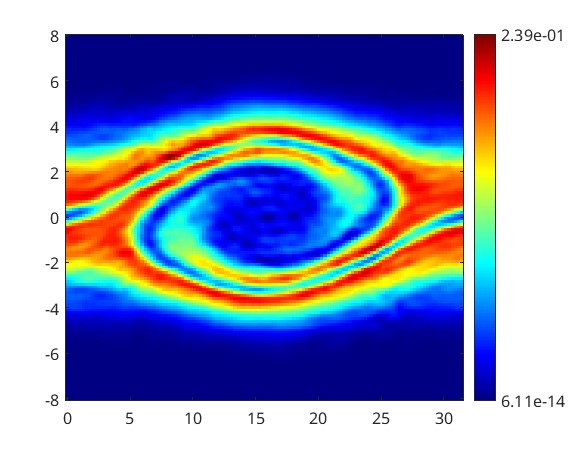}
    \caption{\bP\bP-\bK\bL\ method}
  \end{subfigure}
  \begin{subfigure}[b]{0.45\textwidth}
    \includegraphics[width=\linewidth]{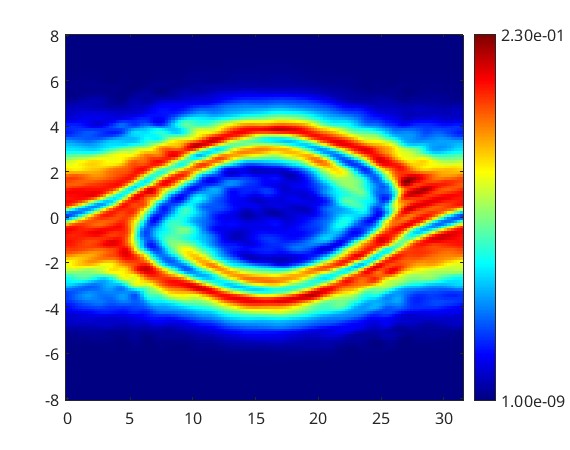}
    \caption{\bM\bM\bS-\bP\bP-\bS\ method}
  \end{subfigure}
  \begin{subfigure}[b]{0.45\textwidth}
    \phantom{\includegraphics[width=\linewidth]{S_T40.jpg}}
  \end{subfigure}
  \caption{Two-stream instability: distribution functions at $t=40$. The color bar indicates the minimum and maximum values of the distribution function. The number of degrees of freedom is $N_x = N_v = 128$.}
  \label{fig:TS_fig3}
\end{figure}

In addition, we show $\Vert\delta\bF\Vert_F$ in Fig.~\ref{fig:TS_fig4} to illustrate the magnitude of the corrections required by both methods. Since the \bP\bP-\bK\bL\ method only targets positivity preservation, Fig.~\ref{fig:TS_fig4} (a) shows that $\Vert\delta\bF\Vert_F$ stays zero until approximately $t=35$, as no correction is needed while the distribution function remains non-negative. Furthermore, $\Vert\delta\bF\Vert_F$ reaches its peak around the time when the minimum value of the distribution function attains its lowest level in the uncorrected case, as illustrated in Fig.~\ref{fig:TS_fig2}. For the \bM\bM\bC-\bP\bP-\bS\ method, the situation is more complicated, as the correction is also required for mass and momentum conservation. In general, we observe that $\Vert\delta\bF\Vert_F$ tends to increase, reaching its maximum around $t=45$, before starting to decrease. Finally, we conclude that the correction needed is rather similar in both cases with the \bP\bP-\bK\bL\ method reaching a larger peak in a few isolated iterations.
\begin{figure}[htbp]
  \centering
    \begin{subfigure}[b]{0.45\textwidth}
    \includegraphics[width=\linewidth]{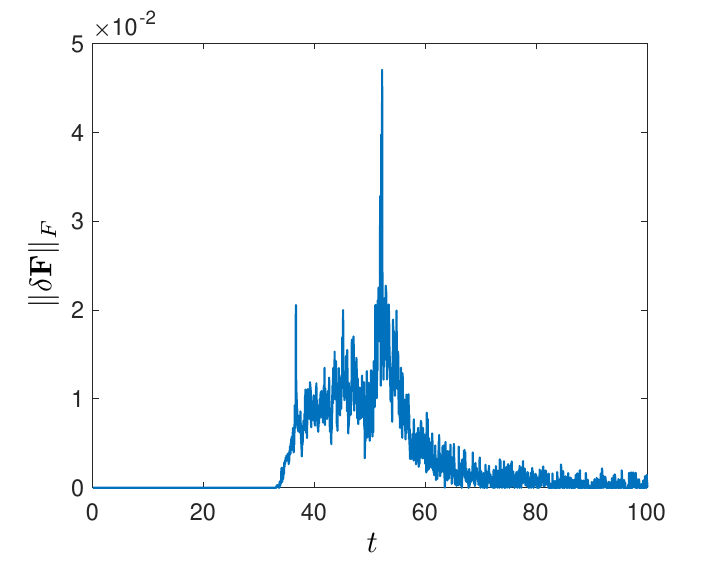}
    \caption{$\bP\bP$-$\bK\bL$ method}
  \end{subfigure}~
    \begin{subfigure}[b]{0.45\textwidth}
    \includegraphics[width=\linewidth]{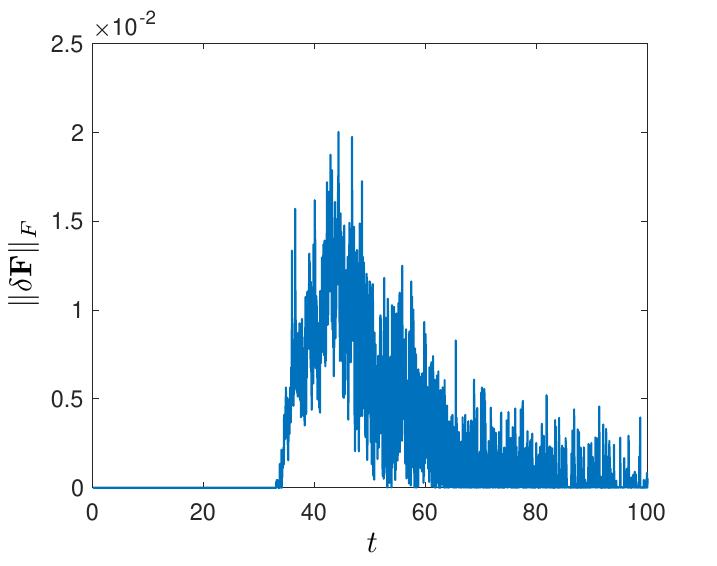}
    \caption{$\bM\bM\bC$-$\bP\bP$-$\bS$ method}
  \end{subfigure}
   \caption{Two-stream instability: the $\Vert\cdot\Vert_F$ of the correction term $\delta\bF$ for \bP\bP-\bK\bL\ and \bM\bM\bC-\bP\bP-\bS\ methods, respectively. The number of degrees of freedom is $N_x = N_v = 128$.}
   \label{fig:TS_fig4}
\end{figure}

\subsection{Computational performance} We have illustrated the effectiveness of the proposed methods in terms of positivity preservation, mass and momentum conservation. Next, we assess their computational performance. In Tab.~\ref{tab:running_times}, we report the run times of all methods for the test cases presented earlier. Since we used $N_x=N_v$, we denote both by $N$ for the remainder of the discussion for simplicity.

Let us break down and analyse the computational cost step by step.  In our workflow, we first check whether the uncorrected solution satisfies the constraints. Therefore, we observe that even when the $\mathbf{PP}$-$\mathbf{KL}$ correction is not activated in the case of Landau damping, there is still a runtime overhead introduced by the $\mathbf{PP}$-$\mathbf{KL}$ method compared to the case without correction.

We also observe a significant increase in runtime when using the \bM\bM\bC-\bP\bP-\bS\ method for the two-stream instability. In particular, the additional runtime introduced by the \bM\bM\bC-\bP\bP-\bS\ method is roughly 50 times larger than that of the \bP\bP-\bK\bL\ method. The computational complexity consideration of the BUG integrator is mostly dominated by the calculation of the QR decompositions with a computational complexity of $O(Nr^2)$.
As mentioned earlier, the complexity of solving each subproblem of \eqref{eq:OP} or \eqref{eq:OP_L} is $O(r^3)$, since each subproblem involves $r$ variables. Hence, the total complexity is $O(Nr^3)$ the \bP\bP-\bK\bL\ method, as there are $N$ subproblems in total. 
On the other hand, for the \bM\bM\bC-\bP\bP-\bS\ method, the complexity of solving \eqref{eq:S_op} is $O(r^6)$. Moreover, Fig~\ref{fig:TS_fig1} (b) shows that $r$ reaches the maximal rank of 20 at a early stage. Therefore, comparing the two complexities gives: $r^6/(Nr^3)=8000/128=62.5$, which explains the substantially larger computational overhead of the \bM\bM\bC-\bP\bP-\bS\ method compared to the \bP\bP-\bK\bL\ method. However, we have to keep in mind that the ratio depends on the ration of the rank $r$ compared to the resolution so that the comparison may change in higher dimensions.
\begin{table}[htbp]
  \caption{Run time (in seconds) for all the experiments presented in this section.}
  \label{tab:running_times}
  \small
  \begin{center}
    \begin{tabular}{c|c|c|c} \hline
      &Without correction&\bP\bP-\bK\bL\ method&\bM\bM\bC-\bP\bP-\bS\ method\\ \hline
      LD&$53.43$&$65.35$&$136.41$\\ \hline
      LD-BGK&$37.45$&$44.73$&$72.06$\\ \hline
      TS&$90.28$&$122.91$&$1861.30$\\ \hline
    \end{tabular}
  \end{center}
\end{table}

\section{Conclusion and outlook}\label{sec:CON} 
In this article, we introduce a novel approach that adds correction terms to low-rank approximations of the Vlasov--Poisson equation such that the corrected solution preserves key properties, including positivity and the conservation of total mass and momentum. The correction terms are determined by solving quadratic programming problems using the active-set method. The proposed methods are designed for general dynamical low-rank schemes under various discretizations, and we demonstrate them using a low-rank framework with spectral methods in space and an explicit Runge-Kutta scheme in time as a showcase. Experiments verify that both the \bP\bP-\bK\bL\ and \bM\bM\bC-\bP\bP-\bS\ methods are positivity-preserving, and that the latter also preserves the conservation of total mass and momentum. The methods introduced in this work can be further exteneded for broader applications. 

There are several limitations of this work that warrant further investigation in future research. First, the purpose of this work is to introduce the proposed correction methods and verify whether positivity preservation and conservation properties are satisfied. The methods are implemented in a straightforward way: the resulting optimization problems are solved using an active-set method implemented in serial MATLAB code. Experiments verify the effectiveness of the proposed methods, but also reveal that they introduce additional computational cost in terms of increased runtime. Future work could focus on improving the computational efficiency of the proposed methods, including through parallelization. Second, the proposed methods are based on the solution of optimization problems, and we employ strategies such as switching between different correction procedures and increasing the rank, to ensure feasibility. However, a rigorous theoretical proof of feasibility and the existence of solutions is still lacking. In general, we do not observe infeasibility with the \bP\bP-\bK\bL\ method, whereas infeasibility may occur for the \bM\bM\bC-\bP\bP-\bS\ method, especially when \(r\) is too small compared to $N_x$ and $N_v$. Finally, this work focuses only on properties that are linear and can be easily formulated as optimization constraints, while ignoring more nontrivial ones. For instance, energy conservation is not considered in this work, since the total energy includes the electric energy, which is not straightforward to incorporate into the optimization constraints. Future work could consider combining our methods with more advanced optimization algorithms, such as the Alternating Direction Method of Multipliers (ADMM).

\section*{Acknowledgement}
We would like to thank Prashant Singh from Uppsala University for valuable discussions.

\bibliographystyle{elsarticle-num}  

\bibliography{ref}
\end{document}